\documentclass[10pt]{article}
\usepackage{amsmath,amssymb,amsthm}
\usepackage{comment,enumerate}
\usepackage{hyperref}
\usepackage{tikz}
\tikzset{every picture/.style={thick}}
\tikzset{every node/.style={draw, circle, fill, inner sep = 2pt}}
\theoremstyle{plain}
\newtheorem{theorem}{Theorem}[section]
\newtheorem{lemma}{Lemma}[section]

\newtheorem{corollary}{Corollary}[section]

\newcommand{\1}{\mathbf{1}}
\begin{document}
\title{Logical aspects of isomorphism of controllable graphs and cospectrality of distance-regularized graphs}
\author{
Aida Abiad\thanks{\texttt{a.abiad.monge@tue.nl},  Department of Mathematics and Computer Science, Eindhoven University of Technology, The Netherlands}
\thanks{Department of Mathematics and Data Science of Vrije, Universiteit Brussel, Belgium} 
\and 
Anuj Dawar\thanks{\texttt{anuj.dawar@cl.cam.ac.uk}, Department of Computer Science and Technology, University of Cambridge, UK}
\and 
Octavio B. Zapata-Fonseca\thanks{\texttt{octavioz@ciencias.unam.mx},  Departamento de Matem\'aticas, Facultad de Ciencias, Universidad Nacional Aut\'onoma de M\'exico, M\'exico}
} 
\date{}
\maketitle
\begin{abstract} 
We consider isomorphism of controllable graphs and cospectrality of distance-regularized graphs (which are known to be distance-regular or distance-biregular) in relation to logical definability. While most characterizations of these equivalence relations for such graph classes are of algebraic and spectral flavor, here we inject tools from first-order logic, extending and unifying several existing results. 
\end{abstract}
\section{Introduction}
\label{sec:introduction}
 The classes of controllable and distance-regularized graphs are relevant objects in algebraic combinatorics, controllability and other fields. 
Most of the known characterizations of such graph classes are of algebraic and spectral flavor, see for instance \cite{DH1,G,L1,LS}. This paper presents the first study of such graph classes making use of first-order logic. 

Distance-regularized graphs were introduced by Godsil and Shawe-Taylor \cite{GS} as a common generalization of distance-regular graphs and  generalized polygons, and they showed that a distance-regularized graph must be either distance-regular or distance-biregular. Here we prove that cospectral distance-regularized graphs are indistinguishable by the counting three-variable first-order logic $C^3$. This result in fact also holds for the so-called pseudo-distance regularized graphs introduced by Fiol and Garriga \cite{FGnew,F2001}, since it is known that they are distance regular or distance-biregular \cite{F2012}. Our first main characterization extends a previous result of this type for strongly regular graphs by Dawar, Severini and Zapata \cite[Proposition 8]{DSZ}, as well as a result for distance-regular graphs by Man{\v{c}}inska,  Roberson and Varvitsiotis \cite[Theorem 7.6]{MRV}, to the more general class of distance-regularized graphs. 
 
Controllable graphs have also attracted a great deal of attention over the last decades. 
Controllable graphs were introduced by Godsil and Severini \cite{GSev} to address the  study continuous-time quantum walks on graphs. O'Rourke and  Touri \cite{OT} proved that almost all graphs are controllable.  
Here we prove that controllable graphs which are indistinguishable by the counting two-variable first-order logic $C^2$ 
are isomorphic, extending results from \cite{DSZ,F,MRV,RS}. 

This paper is structured as follows. 
In Section \ref{sec:preliminaries} we provide the needed definitions and preliminary results. 
In Section \ref{sec:control} we prove a characterization of graph isomorphism for the class of controllable graphs. 
In Section \ref{sec:distanceregularized} we show a characterization of cospectrality of distance-regularized graphs. 

\section{Preliminaries}
\label{sec:preliminaries}

\subsection{First-order logic of graphs}
\label{subsec:logic}
We work in the language of graphs, so there is one binary relation symbol $E$ interpreted as the adjacency relation on the set $X$ of vertices of a graph $\Gamma = (X, E)$. 
The language also contains variables $x,y,z,\dots$  for referring to arbitrary vertices of the graph. 
 We allow any of the usual logical symbols consisting of the standard connectives and quantifiers: $\lnot$ (`not'), $\vee$ (`or'), $\wedge$ (`and'), $\rightarrow$ (`implies'), $\leftrightarrow$ (`if and only if'), $\exists$ (`there exists') and $\forall$ (`for all'). 
Since we only consider undirected graphs without loops or multiple edges, the relation of  adjacency is  irreflexive and symmetric, and hence the sentences 
\begin{align*}
&\forall x \lnot E(x,x)\\
&\forall x \forall y (E(x,y) \leftrightarrow E(y,x))
\end{align*}
 are true in all of our graphs. 
 If the sentence $\phi$ is true in the graph $\Gamma$, then we say that $\Gamma$ \emph{satisfies} $\phi$ and we write $\Gamma \models \phi$. 
 
For any positive integer $k$, we denote by $C^k$ the fragment of first-order logic in which we only have $k$ distinct variables, but we allow counting quantifiers: so for each positive integer $i$, we have a quantifier $\exists^{\geq i}$. 
The semantics is defined so that $\exists^{\geq i}x\varphi$ is true in $\Gamma$ if there are at least $i$ distinct vertices that can be substituted for the variable $x$ to make the formula $\varphi$ true. 
We use the abbreviation $\exists^{=i}x\phi$ for the formula $\exists^{\geq i}x\phi \land \lnot \exists^{\geq i+1}x\phi $ that asserts the existence of exactly $i$ vertices satisfying $\phi$. 
For example, if $d$ is a non-negative integer, then $\Gamma \models \forall x \exists^{=d} y  E(x,y)$ if and only if $\Gamma$ is a $d$-regular graph. 
Thus, we say that the property of being $d$-regular is definable in the logic $C^2$.

We say that two graphs $\Gamma$ and $\Delta$ are $C^k$-\emph{equivalent} (and write $\Gamma \equiv_{C^k} \Delta$) if they satisfy exactly the same $C^k$-sentences: 
 $\Gamma\models \phi$ if and only if $\Delta\models \phi$ for all sentences $\phi$ of $C^k$.
If two graphs are $C^{k+1}$-equivalent for some $k\geq 2$, then (in particular) these two graphs are $C^{k}$-equivalent. 

\subsection{Iterated degree sequences} \label{subsec:C2andisomapprox}
Given a  graph $\Gamma = (X,E)$ and a vertex $x\in X$, let $d(x)$ be the number of vertices which are adjacent to $x$ in $\Gamma$.  
The \emph{degree sequence} of $\Gamma$ is the multiset  $d(\Gamma)=\{d(x)\mid x\in X\}$. 
Similarly, the \emph{iterated degree} of $x$ in $\Gamma$ is defined inductively by $d_0(x)=d(x)$ and $d_{r}(x)=\{d_{r-1}(y)\mid y\in \Gamma(x)\}$ for every $r >0$.
The \emph{iterated degree sequence} of $\Gamma$ is the multiset $D_\Gamma =\{d_r(\Gamma)\mid r\geq 0\}$ with $d_0(\Gamma)=d(\Gamma)$ and $d_{r}(\Gamma)=\{d_{r-1}(x)\mid x\in X\}$ for $r>0$.

The process of finding the iterated degree sequence of a graph has several widely adopted names; it is known as \emph{canonical labelling}, \emph{color refinement}, \emph{naive vertex classification} or \emph{1-dimensional Weisfeiler--Leman algorithm}. 
Immerman and Lander proved that $C^2$-equivalence  is a  necessary and sufficient condition for having the same iterated degree sequence (see \cite[Theorem 4.8.1]{IL}). 

\begin{theorem}[\cite{IL}]\label{thm:c2}
	Two graphs  are  $C^2$-equivalent if and only if they  have equal iterated degree sequences. 
\end{theorem}

Two regular graphs with the same number of vertices and the same degree necessarily have the same iterated degree sequence. 
For example, if $\Gamma$ is the disjoint union of two triangles and $\Delta$ is the cycle of length 6, then both $\Gamma$ and $\Delta$ have $6$ vertices and degree $2$, and hence their iterated degree sequence is the same.  

\subsection{Fractional isomorphism}
Let $\Gamma = (X,E)$ be a finite graph. 
The \emph{adjacency matrix} of $\Gamma$ is the square matrix $A$ indexed by the vertex-set $X$, whose entries $A_{xy}$ are defined as $A_{xy} = 1$ if the vertices $x$ and $y$ are adjacent, and $A_{xy} = 0$ otherwise.

A real matrix $S$ is called \emph{doubly stochastic} if all its entries are non-negative and every row and every column sums to 1.  
The Birkhoff--von Neumann theorem says that the set of all $n\times n$ doubly stochastic matrices is a compact and convex set whose extreme points are the permutation matrices. 
Two graphs $\Gamma$ and $\Delta$ with adjacency matrices $A$ and $B$ are \emph{fractionally isomorphic} if there exists a doubly stochastic matrix $S$ such that $SA = BS$.   

\begin{theorem}[\cite{RSU}]\label{thm:degree}
	Two graphs are fractionally isomorphic if and only if they have the same iterated degree sequence. 
\end{theorem}

\subsection{Graph spectra}
Let $\Gamma$ be a finite graph with adjacency matrix $A$. 
By definition, the matrix $A$ is symmetric, i.e. $A^\top = A$. 
Therefore, all its eigenvalues are real. 
The \emph{spectrum} of $\Gamma$ is the multiset of all eigenvalues of $A$ counted with their multiplicities.  
We say that two graphs are \emph{cospectral} if they have the same spectrum. 

\begin{theorem}[\cite{DSZ}]\label{thm:C3}
    If two graphs are $C^3$-equivalent, then they are cospectral. 
\end{theorem}

In general, cospectrality is not a sufficient condition for $C^3$-equivalence.  
For example, the complete bipartite graph $K_{1,4}$ and the disjoint union $C_4 + K_1$ of the 4-cycle $C_4$ and a vertex $K_1$ are cospectral but  not $C^3$-equivalent. 
Indeed, they are distinguishable by a sentence of $C^3$ that asserts the existence of an isolated vertex: $K_{1,4}\not \models \exists x \forall y \lnot E(x,y)$  and $C_4 + K_1\models\exists x \forall y \lnot E(x,y)$.

Also, cospectrality is not necessary if we allow only two variables: the $6$-cycle $C_6$ and  the disjoint union  $2K_3$ of two triangles $K_3$ are not cospectral and cannot be distinguished by any sentence using two distinct variables and counting quantifiers (see, e.g., \cite[Section 4.7]{IL}). 

Even if we allow more than three variables, the use of counting quantifiers is essential: for each positive integer $k\geq 3$, there is a pair of non-isomorphic graphs which are not  cospectral  and that cannot be distinguished by any sentence using  $k$ distinct variables and no counting quantifiers \cite[Proposition 4]{DSZ}.  

\subsection{Coherent configurations} 
\label{subsec:coherent}
A \emph{coherent configuration} on a set $X$ is a collection $\Pi=\{R_1,\dots,R_s\}$ of binary relations  on $X$ satisfying the following conditions:
\begin{enumerate}[(i)]
\item $\{R_1,\dots,R_s\}$ is a partition of $X\times X$;
\item there is a subset $H$ of $\{1,\dots,s\}$ such that $\bigcup_{h\in H} R_h = \{(x,x) \mid x\in X\}$;
\item  for each $i \in \{1,\dots, s\}$ there exists a $j \in \{1,\dots, s\}$ such that $(x,y)\in R_i$ implies $(y,x)\in R_j$;
\item for each $i,j,k \in \{1,\dots, s\}$ there exists a non-negative integer $p_{ij}(k)$ such that if $(x,y)\in R_k$, the number of elements $z \in X$ with $(x,z)\in R_i$ and $(z,y)\in R_j$ is equal to $p_{ij}(k)$.
\end{enumerate}
The numbers $p_{ij}(k)$ are called the \emph{intersection numbers} of the coherent configuration $\Pi$. The number $s$ of relations is called the \emph{rank} of $\Pi$. 

The set of all coherent configurations on $X$ is partially ordered by refinement: $\Theta$ \emph{refines} $\Lambda$ if every $R\in \Theta$ is contained in some $S\in \Lambda$. 

\subsubsection{The coherent configuration of a graph}
Let $\Gamma = (X,E)$ be a  graph and $\Pi_\Gamma = \{R_1,\dots,R_s\}$ be a collection of binary relations on $X$ 
satisfying (i)--(iv) above, and 
\begin{itemize}
\item[(v)] there is a subset $G$ of $\{1,\dots,s\}$ such that $\bigcup_{g\in G} R_g = E$.
\end{itemize}
Then $\Pi_\Gamma$ is the unique maximal coherent configuration on $X$ such that $E$ is the union of some relations in $\Pi_\Gamma$.

The \emph{2-dimensional Weisfeiler--Leman algorithm} is an efficient procedure that computes $\Pi_\Gamma$ for any finite graph $\Gamma$. 
This procedure works as follows. 
Let $f_0$ be the function on $X\times X$ defined by
\[
f_0(x,y) = 
\begin{cases}
0 & \textnormal{if  $x = y$}, \\
1 & \textnormal{if $(x,y)\in E$}, \\
2 & \textnormal{otherwise}. 
\end{cases}
\]
For $r\geq 0$, define the multiset $f_{r+1}(x,y) = \{( f_r(x,z) ,f_r(z,y)) \mid z\in X\}$ and let $\Pi_r$ be the partition of $X\times X$  so that $(x,y)$ and $(u,v)$ are in the same part if and only if  $f_r(x,y) = f_r(u,v)$. 
(Note that $f_1(x,y) \neq f_1(u,v)$ if $(x,y)$ and $(u,v)$ are at distinct mutual distance $d(x,y)\neq d(u,v)$.)
Let $\Pi_\Gamma = \Pi_r$ when $\Pi_r = \Pi_{r+1}$.  

Figure \ref{fig:partition} shows the coherent configuration obtained by this algorithm  for the 8-cycle $C_8$. 
In this case, all the relations in the partition are symmetric.  
\begin{figure}[htp!]
\medskip\hfil\begin{tikzpicture}[scale=0.85]
\node (a1) at ( 45:1) {};
\node (a2) at ( 90:1) {};
\node (a3) at (135:1) {};
\node (a4) at (180:1) {};
\node (a5) at (225:1) {};
\node (a6) at (270:1) {};
\node (a7) at (315:1) {};
\node (a8) at (360:1) {};
\draw (a1) to [loop above] (a1);
\draw (a2) to [loop above] (a2);
\draw (a3) to [loop above] (a3);
\draw (a4) to [loop above] (a4);
\draw (a5) to [loop above] (a5);
\draw (a6) to [loop above] (a6);
\draw (a7) to [loop above] (a7);
\draw (a8) to [loop above] (a8);
\end{tikzpicture}\hfil
\begin{tikzpicture}[scale=0.85]
\node (a1) at ( 45:1) {};
\node (a2) at ( 90:1) {};
\node (a3) at (135:1) {};
\node (a4) at (180:1) {};
\node (a5) at (225:1) {};
\node (a6) at (270:1) {};
\node (a7) at (315:1) {};
\node (a8) at (360:1) {};
\draw (a1)--(a2)--(a3)--(a4)--(a5)--(a6)--(a7)--(a8)--(a1);
\end{tikzpicture}\hfil\begin{tikzpicture}[scale=0.85]
\node (a1) at ( 45:1) {};
\node (a2) at ( 90:1) {};
\node (a3) at (135:1) {};
\node (a4) at (180:1) {};
\node (a5) at (225:1) {};
\node (a6) at (270:1) {};
\node (a7) at (315:1) {};
\node (a8) at (360:1) {};
\draw (a1)--(a3)--(a5)--(a7)--(a1); 
\draw (a2)--(a4)--(a6)--(a8)--(a2);
\end{tikzpicture}\hfil\begin{tikzpicture}[scale=0.85]
\node (a1) at ( 45:1) {};
\node (a2) at ( 90:1) {};
\node (a3) at (135:1) {};
\node (a4) at (180:1) {};
\node (a5) at (225:1) {};
\node (a6) at (270:1) {};
\node (a7) at (315:1) {};
\node (a8) at (360:1) {};
\draw (a1)--(a4)--(a7)--(a2)--(a5)--(a8)--(a3)--(a6)--(a1); 
\end{tikzpicture}\hfil\begin{tikzpicture}[scale=0.85]
\node (a1) at ( 45:1) {};
\node (a2) at ( 90:1) {};
\node (a3) at (135:1) {};
\node (a4) at (180:1) {};
\node (a5) at (225:1) {};
\node (a6) at (270:1) {};
\node (a7) at (315:1) {};
\node (a8) at (360:1) {};
\draw (a1)--(a5);
\draw (a2)--(a6);
\draw (a3)--(a7);
\draw (a4)--(a8); 
\end{tikzpicture}
\medskip
\caption{Coherent configuration of rank $5$ determined by the $8$-cycle.}
\label{fig:partition}
\end{figure}
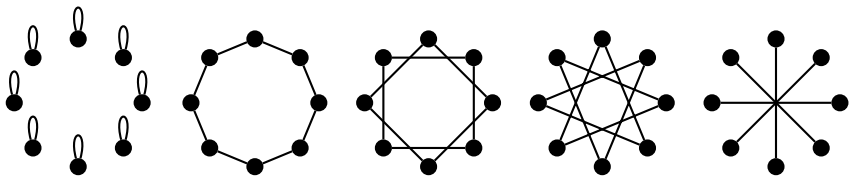

\subsubsection{Intersection equivalence}
We say that two coherent configurations $\Pi$ and $\bar{\Pi}$  with intersection numbers $p_{ij}(k)$ and $\bar{p}_{ij}(k)$ are \emph{intersection equivalent} if they have the same rank $s$, and have the same intersection numbers (up to a permutation of the indices): 
\[p_{ij}(k) = \bar{p}_{i j}(k)\]
for all $i,j,k\in \{1,\dots,s\}$.

There is a classic result of Cai, F\"urer and Immerman  that relates the  $C^3$-equivalence of graphs with the intersection equivalence of the coherent configurations obtained by the 2-dimensional Weisfeiler--Leman algorithm (see, e.g.,  \cite[Theorem 5.2]{CFI}). 
In our context, the right formulation is the following. 

\begin{theorem}[\cite{CFI}]\label{thm:cfi} 
Two graphs are $C^3$-equivalent if and only if they have intersection equivalent coherent configurations. 
\end{theorem}

\section{Isomorphism of controllable graphs}
\label{sec:control}
The theory of controllable graphs was developed by Godsil in \cite{G}, where it was conjectured that the proportion of graphs on $n$ vertices which are controllable goes to 1 as $n \to \infty$. 
It was later confirmed by O'Rourke and  Touri \cite{OT} that indeed almost all graphs are controllable. 
In this section we show that if two controllable graphs are $C^{2}$-equivalent, then these two graphs  are isomorphic. 

\subsection{Controllable graphs}
The \emph{walk matrix} of a graph $\Gamma = (X,E)$ with $n$ vertices and adjacency matrix $A$ is the $n\times n$ matrix 
\[
    W_\Gamma =  \big[ \1\quad A\1\quad \cdots \quad A^{n-1}\1 \big],
\] 
where $\1$ denotes the all-one vector. 
If $X=\{1,\dots,n\}$, the $ij$-entry of the walk matrix $W_\Gamma$ counts the number of walks in $\Gamma$ of length $j-1$ starting at vertex $i$.  
 We say that the graph is \emph{controllable} if its walk matrix is invertible.  

If $\Gamma$ is regular of degree $d$, then $A\1 = d\1$; this implies that $W_\Gamma$ has rank 1. 
It follows that controllable graphs cannot be regular. 

We note also that if $P$ is a permutation matrix that commutes with $A$, then $PA^i\1 = A^i P\1 = A^i\1$ for $0\leq i\leq n-1$, and hence $PW_\Gamma = W_\Gamma$; this implies that $P=I$ when $W_\Gamma$ is invertible.  
Therefore, the only automorphism of a controllable graph is the identity. 

\subsection{Walk-equivalence} \label{subsec:C2andwalkequivalence}
We say that two graphs are \emph{walk-equivalent} if they have the same number of walks of length $i$ for all $i\geq 0$.

\begin{lemma} \label{lem:c2}
   If two graphs are $C^2$-equivalent, then they are walk-equivalent.
\end{lemma}   
\begin{proof}
    By Theorem \ref{thm:c2} and Theorem \ref{thm:degree}, $C^2$-equivalent graphs are fractionally isomorphic. 
    Let $\Gamma$ and $\Delta$ be graphs with adjacency matrices $A$ and $B$, respectively. 
    Since the sum of entries of the $i$-th power of the adjacency matrix is equal to the number of walks of length $i$ in the graph, $\1^\top A^i \1$ is the number of walks of length $i$ in $\Gamma$ and $\1^\top B^i \1$ is the number of walks of length $i$ in $\Delta$. 
    If $\Gamma$ and $\Delta$ are fractionally isomorphic, then there is a doubly stochastic matrix $S$ such that $SA = BS$ and hence 
    \[\1^\top A^i \1 = \1^\top S A^i \1 = \1^\top B^i S \1 = \1^\top B^i \1.\qedhere\]
\end{proof}

\begin{lemma} \label{lem:walk}
    If the graphs $\Gamma$ and $\Delta$ are $C^2$-equivalent, then there exists a permutation matrix $P$ such that $PW_\Gamma = W_{\Delta}$.
\end{lemma}   
\begin{proof} 
For each $h\geq 1$ and $i\geq 0$, we define a $C^2$-formula  $\phi^{h}_{i}(x)$ that
asserts the existence of exactly $h$ walks of length $i$ starting at vertex $x$.
For any vertex $x$, there is exactly one walk of length 0 starting at $x$, so letting $\bot$ be any false statement and $\top$ be any true statement, we define the $C^2$-formulas $\phi^{h}_{i}(x)$ by induction on $i$.
If $i = 0$, then 
\[
\phi^0_{0}(x) := \bot, \qquad \phi^1_{0}(x) := \top
 \qquad \textnormal{and} \qquad 
\phi^h_{0}(x) := \bot \quad \textnormal{for}\ h>1.
\] 
Now if $i = 1$, then 
\[
\phi^0_{1}(x) := \forall y\lnot E(x,y) \qquad \textnormal{and} \qquad 
\phi^h_{1}(x) := \exists^{=h} y  E(x,y) \quad \textnormal{for}\ h>0.
\]
For $i > 1$, we define
\[
\phi^0_{i+1}(x) := \forall y ( E(x,y) \to \phi^0_{i}(y) ),
\]
and if $h >0$, then 
\[
\phi^h_{i+1}(x) := \bigvee_{(h_1^{a_1},\dots,h_r^{a_r})\in H} [ ( \bigwedge_{j = 1}^r \exists^{=a_j}y\ \phi_{i}^{h_j}(y)) \land \exists^{\geq a}y E(x,y)],
\]
where $H$ denote the set of all integer partitions of $h$ (i.e., $h_j\geq 0$, $a_j\geq 1$ and $h= \sum_{j=1}^r a_jh_j$), and $a = \sum_{j=1}^r a_j$. 

We observe that in all these definitions we do not use more than two distinct variables.

By definition,  $\Gamma\models  \phi_{i}^{h}(x)$ if and only if there are $h$ walks of length $i$ in $\Gamma$ starting at vertex $x$. 

Since $\Gamma$ and $\Delta$ are $C^2$-equivalent, 
there is a vertex $x$ of $\Gamma$ such that $\Gamma\models \phi^{h}_{i}(x)$ if and only if there is a vertex $\bar{x}$ of $\Delta$ such that $\Delta \models \phi^{h}_{i}(\bar{x})$. 
Hence the mapping $x \mapsto \bar{x}$ is a bijection between the sets $\{x \in V_\Gamma\mid \Gamma\models\phi^{h}_{i}(x)\}$ and $\{\bar{x} \in V_{\Delta}\mid \Delta \models\phi^{h}_{i}(\bar{x})\}$. 
Since the rows of the walk matrix of a graph are indexed by the vertices of the graph, it follows that  the above bijection determines a permutation matrix $P$ such that $PW_{\Gamma} = W_{\Delta}$. 
\end{proof}

\subsection{Isomorphism and $C^2$-equivalence}
The \emph{complement} of a graph with adjacency matrix $A$ is the graph with adjacency matrix $J-I-A$. 
We say that two graphs are \emph{generalized cospectral} if they are cospectral with cospectral complements. 
A necessary condition for generalized cospectrality is walk-equivalence.
It follows from \cite[Corollary 3.2]{G} that, for controllable graphs, this condition is also sufficient. 

\begin{theorem}[\cite{G}]\label{thm:sufficient}
Two controllable graphs are walk-equivalent if and only if they are generalized cospectral. 
\end{theorem}

If two controllable graphs $\Gamma$ and $\Delta$ with adjacency matrices $A$ and $B$ are generalized cospectral, then the matrix $Q = W_{\Delta}W_{\Gamma}^{-1}$ satisfies $QAQ^\top= B$ and $Q\1 = \1$ (see \cite[Lemma 6.1]{G}).
We shall use this remark to prove our next result. 

\begin{theorem} \label{thm:iso}
Two controllable graphs $\Gamma$ and $\Delta$ are isomorphic if and only if they are  $C^2$-equivalent.
\end{theorem}
\begin{proof}  
Let $A$ and $B$ be the adjacency matrices of  $\Gamma$ and $\Delta$, respectively. 
Suppose $\Gamma$ and $\Delta$ are  $C^2$-equivalent. 
From Lemma \ref{lem:c2},  $\Gamma$ and $\Delta$ are walk-equivalent.
Theorem \ref{thm:sufficient} implies that $\Gamma$ and $\Delta$ are generalized cospectral. 
Then the matrix $Q = W_{\Delta}W_{\Gamma}^{-1}$ satisfies $QAQ^\top= B$ and $Q\1 = \1$.  
Lemma  \ref{lem:walk} implies that there is a permutation matrix $P$ such that $PW_{\Gamma} =  W_{\Delta}$. 
It  follows that \[Q  = W_{\Delta}W_{\Gamma}^{-1} = PW_{\Gamma}W_{\Gamma}^{-1} = P,\]  and hence that $PAP^\top = B$. 
Therefore,  $\Gamma$ and $\Delta$ are isomorphic and this proves the theorem, because isomorphic graphs are $C^k$-equivalent for any $k>0$. 
\end{proof}

\section{Cospectrality of distance-regularized graphs}
\label{sec:distanceregularized}
Distance-regularized graphs are either distance-regular or distance-biregular \cite{GS}. 
In this section, we prove that the coherent configurations determined by cospectral distance-biregular graphs are intersection equivalent, see Theorem \ref{thm:cospec} below. 
As a consequence we obtain that if two distance-regularized are cospectral, then these two graphs are $C^3$-equivalent. 

\subsection{Distance-regularized graphs }
Given a graph $\Gamma = (X,E)$ and a vertex $x\in X$,  let $\Gamma_i(x)$ be the set of vertices at distance $i\geq 0$ from $x$ in $\Gamma$.
If $\Gamma$ is connected with diameter $d$, 
for any two vertices $x,y\in X$ with $y \in \Gamma_i(x)$, we define 
\[\begin{aligned}
c_i(x,y)&=|\Gamma_{i-1}(x)\cap \Gamma_1(y)|,\\
a_i(x,y)&=|\Gamma_{i}(x)\cap \Gamma_1(y)|,\\
b_i(x,y)&=|\Gamma_{i+1}(x)\cap \Gamma_1(y)|,
\end{aligned}\]
with $c_i(x,y) = 0$ when $i=0$. 
We say that the vertex $x$ is \emph{distance-regularized} if  the numbers $c_i(x,y), a_i(x,y)$ and $b_i(x,y)$ are independent of the choice of the vertex $y$ in $\Gamma_i(x)$, and in this case we denote them by $c_{i}(x), a_{i}(x)$ and $b_{i}(x)$, respectively.
 
A \emph{distance-regularized graph} is a connected graph in which every vertex is distance-regularized. 
If $\Gamma$ is  distance-regularized with diameter $d$, the array
\[
\iota(x)
=
\left[\begin{array}{ccccc}
0		&	c_1(x)	&	\cdots	&	c_{d-1}(x)	&	c_d(x)\\
a_0(x)		&	a_1(x)	&	\cdots	&	a_{d-1}(x)	&	a_d(x)\\
b_0(x)	&	b_1(x)	&	\cdots	&	b_{d-1}(x)	&	0
\end{array}\right]
\]
is called the \emph{intersection array} of vertex $x$.  

\subsubsection{Distance-regular graphs}
A \emph{distance-regular graph} is a distance-regularized graph in which all vertices have the same intersection array. 
If $\Gamma$ is  distance-regular with diameter $d$, the numbers $c_{i}(x), a_{i}(x)$ and $b_{i}(x)$ are denoted $c_{i}, a_{i}$ and $b_{i}$ ($0\leq i\leq d$). 
By definition, $\Gamma$ is $k$-regular with $k = b_0$ and $a_i = k - b_i - c_i $ for all $i$. 
The notation used in the literature for the intersection array of a distance-regular graph of diameter $d$ is 
\[
\iota = \{b_0,b_1,\dots,b_{d-1} ; c_1,c_2,\dots, c_d\}.
\]
See, e.g., the monograph by Brouwer, Cohen and Neumaier~\cite{BCN}. 

Let $\Gamma = (X,E)$ be a distance regular-graph with diameter $d$. 
For $i\geq 0$, we define the symmetric relation $R_{i}$ on $X$ by
\[
R_{i} = \{(x,y)\in X\times X \mid  d(x,y) = i\}.
\]
Then $R_{i} = \emptyset$ for $i>d$ and $R_i\cap R_j = \emptyset$ for all $i,j$. 
Moreover,   
\[
R_0\cup R_1 \cup \cdots \cup R_d = X\times X
\]
and hence $\{R_0,R_1,\dots,R_d\}$ is a partition of $X\times X$. 
Furthermore, we have $R_0= \{(x,x) \mid  x\in X\}$ and $R_1 = E$. 
Since $\Gamma$ is distance regular, there exist numbers $p_{ij}(k)$ such that for all $(x,y)\in R_k$,  
\[
 |\Gamma_i(x)\cap \Gamma_j(y)|  = p_{ij}(k). 
\] 
Therefore,  $\Pi_\Gamma = \{R_0,R_1,\dots,R_d\}$ is the coherent configuration of $\Gamma$. 

The next result says that the intersection numbers $p_{ij}(k)$ of $\Pi_\Gamma$ are determined by the intersection array of $\Gamma$ (see \cite{BCN}, Lemma 4.1.7). 

\begin{lemma}[\cite{BCN}]
\[
p_{i0}(k) = \delta_{ik}, \qquad p_{0j}(k)  = \delta_{jk},
\]
\[
p_{i1}(k) =
\begin{cases}
    c_i & \textnormal{if $k = i+1$},\\
    a_i & \textnormal{if $k = i$},\\
    b_i & \textnormal{if $k = i-1$},\\
    0 & \textnormal{otherwise}.
\end{cases} 
\] 
\[
p_{ij+1}(k)
= \frac{1}{c_{j+1}} 
[p_{i-1 j}(k)b_{i-1} + p_{ij}(k)(a_i - a_j) + p_{i+1j}(k)c_{i+1} - p_{ij-1}(k) b_{j-1}
].
\]
\end{lemma}
\subsubsection{Distance-biregular graphs}
\label{subsec:biregular}
A \emph{distance-biregular} graph is a bipartite distance-regularized graph in which all vertices in the same class of the bipartition have the same intersection array.
If $\Gamma=(X,E)$ is distance-biregular with bipartition $(X',X'')$, the intersection arrays are
\[
\iota' = \{b'_0,\dots,b'_{d'-1} ; c'_1,\dots, c'_{d'}\}\quad\textnormal{ and }\quad
\iota'' = \{b''_0,\dots,b''_{d''-1} ; c''_1,\dots, c''_{d''}\}
\]
with $d'=\max\{d(x,y) \mid x\in X', y\in X\}$ and $d''=\max\{d(x,y) \mid x\in X'', y\in X\}$.
By definition, $\Gamma$ is $(k,\ell)$-semiregular with $k = b'_0$ and $\ell = b''_0$. 
Moreover,  for all $i\geq 0$, we have  
\[
b'_i + c'_i =
\begin{cases}
k & \textnormal{if $i$ is even},\\
\ell & \textnormal{if $i$ is odd},
\end{cases}
\quad\textnormal{ and }\quad
b''_i + c''_i =
\begin{cases}
\ell&\textnormal{if $i$ is even},\\
k&\textnormal{if $i$ is odd}.
\end{cases}
\]

Figure \ref{fig:subdivision} shows an example of a distance-biregular graph which is not distance-regular.  
This graph is a subdivision of the complete graph $K_4$.   
It is distance-biregular with intersection arrays
\[
\iota' = \{3,1,2; 1,1,2\} \quad \textnormal{and} \quad \iota'' = \{2,2,1,1; 1,1,2,2\}.
\]
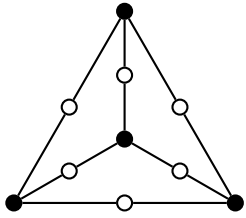
\begin{figure}[htp!]
\tikzset{every node/.style={draw, circle, inner sep = 2pt}}
\medskip\hfil\begin{tikzpicture}[scale=0.85]
\node[fill] (a1) at (90:0) {};
\node (a2) at (90:1) {};
\node[fill] (a3) at (90:2) {};
\node (a4) at (210:1) {};
\node[fill] (a5) at (210:2) {};
\node (a6) at (330:1) {};
\node[fill] (a7) at (330:2) {};
\node (a8) at (150:1) {};
\node (a9) at (270:1) {};
\node (a10) at (30:1) {};
\draw (a1)--(a2)--(a3);
\draw (a1)--(a4)--(a5);
\draw (a1)--(a6)--(a7);
\draw (a3)--(a8)--(a5)--(a9)--(a7)--(a10)--(a3);
\end{tikzpicture}
\medskip
\caption{A distance-biregular graph which is not distance-regular.}
\label{fig:subdivision}
\end{figure}

\subsection{Cospectrality}\label{sec:cospec}
The intersection arrays of a distance-biregular graph determine and are determined by the spectrum and the degrees of semiregularity (see \cite[Theorem 2.7.2]{L2}). 

\begin{theorem}[\cite{L2}]\label{thm:sabrina}
Let $\Gamma$ and $\Delta$ be distance-biregular graphs, and let $\iota'_\Gamma, \iota''_\Gamma$ and $\iota'_\Delta,\iota''_\Delta$ be their intersection arrays. 
Then, $\Gamma$ and $\Delta$ are cospectral and have the same degrees of semiregularity if and only if $\iota'_\Gamma = \iota'_\Delta$ and $\iota''_\Gamma =\iota''_\Delta$. 
\end{theorem}

This result together with the next lemma imply that cospectral distance-biregular graphs have the same intersection arrays.

\begin{lemma}\label{lem:semi} 
Let $\Gamma$ be a bipartite $(k,\ell)$-semiregular graph and $\Delta$ a bipartite $(\bar{k},\bar{\ell})$-semiregular graph. 
If $\Gamma$ and $\Delta$ are cospectral and $k\neq \bar{k}$ or $\ell \neq \bar{\ell}$, then $k = \bar{\ell}$ and $\ell = \bar{k}$.
\end{lemma}
\begin{proof}
Let $\Gamma=(X,E)$ be a bipartite $(k,\ell)$-semiregular graph with bipartition $(X',X'')$, and let $p =|X'|$ and $q = |X''|$.
The adjacency matrix $A$ of $\Gamma$ can be written in the form
\[
A = 
\begin{bmatrix}
    0 & N \\
    N^\top & 0   
\end{bmatrix},
\]
where $N$ is a $p\times q$ matrix with rows and columns indexed by the vertices in $X'$ and $X''$, respectively.
If we write $\1$ for a vector of suitable size whose coordinates are all one, it follows that $N\1 = k\1$ and $N^\top \1 = \ell \1$, and so that $k\ell$ is an eigenvalue of both $NN^\top$ and $N^\top N$. 
Since 
\[
A^2 = 
\begin{bmatrix}
    NN^\top 	& 0 \\
    0		& N^\top N
\end{bmatrix},
\]
we see that $k\ell$ is an eigenvalue of $A^2$, and hence that $\sqrt{k\ell}$ and $-\sqrt{k\ell}$ are both eigenvalues of $A$. 
Since all the entries of $\1$ are positive, by the Perron--Frobenius theorem,  $\sqrt{k\ell}$ is the largest eigenvalue of $A$. 

Let $\Gamma$ and $\Delta$ both have spectrum $\theta_0^{m_0}, \dots, \theta_d^{m_d}$ where $\theta_0 > \cdots > \theta_d$ and the exponent $m_i$ denotes the multiplicity of the eigenvalue $\theta_i$. 
Then 
\begin{equation}
\theta_0 = \sqrt{k\ell} = \sqrt{\bar{k}\bar{\ell}},
\end{equation}
and so 
\begin{equation} \label{eq:0}
k\ell = \bar{k}\bar{\ell}. 
\end{equation}

Suppose 
$\Delta$ has bipartition $(\bar{X}',\bar{X}'')$ with $|\bar{X}'| = \bar{p}$ and $|\bar{X}''| = \bar{q}$. 
The sum of the multiplicities is equal to the number of vertices, and the sum of the squares of  the eigenvalues is equal to two times the number of edges. 
Then 
\begin{equation} \label{eq:1}
p+q = \sum_{i=0}^d m_i = \bar{p} + \bar{q}
\end{equation}
and 
\begin{equation} \label{eq:2}
kp = \ell q = \frac{1}{2}\sum_{i=0}^d \theta_i^2 =  \bar{k}\bar{p} = \bar{\ell}\bar{q}.
\end{equation}

Now suppose without loss of generality that $\ell \neq \bar{\ell}$. 
From \ref{eq:0} and  \ref{eq:1}, we have $k\ell(p + q) = \bar{k}\bar{\ell}(\bar{p} + \bar{q})$.  
Then, using \ref{eq:2}, we get \[kp\ell + k\ell q = \bar{k}\bar{p}\bar{\ell} + \bar{k}\bar{\ell}\bar{q} = kp\bar{\ell} + \bar{k}\ell q.\] 
It follows, again by \ref{eq:2}, that \[kp(\ell-\bar{\ell}) = \ell q(\bar{k} -k) = kp(\bar{k} -k),\] so $0 \neq \ell-\bar{\ell} = \bar{k} -k$. 
Then, using \ref{eq:0} again, we have \[(\ell-\bar{\ell})\ell = (\bar{k} -k)\ell = \bar{k}\ell -k\ell = \bar{k}\ell - \bar{k}\bar{\ell} = \bar{k}(\ell-\bar{\ell}),\] so $\ell = \bar{k}$. 
Thus $\ell\bar{\ell} = \bar{k}\bar{\ell} = k\ell$, and so $\bar{\ell} = k$.
\end{proof}

Now we see that for distance-biregular graphs, cospectrality is equivalent to having the same intersection arrays. 

\begin{corollary}\label{cor:arrays}
 Two distance-biregular graphs are cospectral if and only if they have the same intersection arrays. 
\end{corollary}
\begin{proof}
If $\Gamma$ and $\Delta$ are cospectral distance-biregular graphs, then Lemma \ref{lem:semi} implies that $\Gamma$ and $\Delta$ have the same degrees of semiregularity. Thus, by Theorem \ref{thm:sabrina}, $\Gamma$ and $\Delta$ have the same intersection arrays. 
Conversely, if two distance-biregular graphs  have the same intersection arrays, then Theorem \ref{thm:sabrina} implies that they are cospectral. 
\end{proof}

\subsection{Intersection equivalence and $C^3$-equivalence}
In general, a coherent configuration consists of a set of directed graphs. 
Sometimes some of these graphs are undirected.   
For instance, the coherent configuration of the subdivision of $K_4$ is presented in Figure \ref{fig:coherent}.

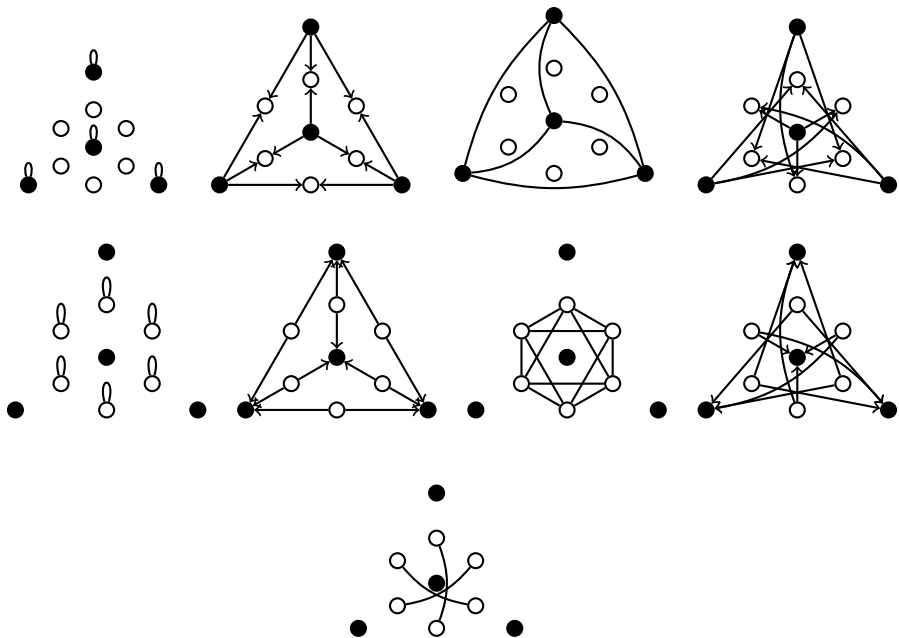
\begin{figure}[htp!]
\tikzset{every node/.style={draw, circle, inner sep = 2pt}}
\medskip\hfil\begin{tikzpicture}[scale=0.5]
\node[fill] (a1) at (90:0) {};
\node (a2) at (90:1) {};
\node[fill] (a3) at (90:2) {};
\node (a4) at (210:1) {};
\node[fill] (a5) at (210:2) {};
\node (a6) at (330:1) {};
\node[fill] (a7) at (330:2) {};
\node (a8) at (150:1) {};
\node (a9) at (270:1) {};
\node (a10) at (30:1) {};
\draw (a1) to [loop above] (a1); 
\draw (a3) to [loop above] (a3); 
\draw (a5) to [loop above] (a5); 
\draw (a7) to [loop above] (a7); 
\end{tikzpicture}
\medskip\hfil\begin{tikzpicture}[scale=0.7]
\node[fill] (a1) at (90:0) {};
\node (a2) at (90:1) {};
\node[fill] (a3) at (90:2) {};
\node (a4) at (210:1) {};
\node[fill] (a5) at (210:2) {};
\node (a6) at (330:1) {};
\node[fill] (a7) at (330:2) {};
\node (a8) at (150:1) {};
\node (a9) at (270:1) {};
\node (a10) at (30:1) {};
\draw[->] (a3)--(a2);
\draw[->] (a5)--(a4);
\draw[->] (a7)--(a6);
\draw[->] (a3)--(a8);
\draw[->] (a5)--(a8);
\draw[->] (a5)--(a9);
\draw[->] (a7)--(a9);
\draw[->] (a7)--(a10);
\draw[->] (a3)--(a10);
\draw[->] (a1)--(a2);
\draw[->] (a1)--(a4);
\draw[->] (a1)--(a6);
\end{tikzpicture}
\medskip\hfil\begin{tikzpicture}[scale=0.7]
\node[fill] (a1) at (90:0) {};
\node (a2) at (90:1) {};
\node[fill] (a3) at (90:2) {};
\node (a4) at (210:1) {};
\node[fill] (a5) at (210:2) {};
\node (a6) at (330:1) {};
\node[fill] (a7) at (330:2) {};
\node (a8) at (150:1) {};
\node (a9) at (270:1) {};
\node (a10) at (30:1) {};
\draw (a1) to [bend left=25] (a3); 
\draw (a1) to [bend left=25] (a5); 
\draw (a1) to [bend left=25] (a7); 
\draw (a3) to [bend left=15] (a7); 
\draw (a7) to [bend left=15] (a5); 
\draw (a5) to [bend left=15] (a3); 
\end{tikzpicture}
\medskip\hfil\begin{tikzpicture}[scale=0.7]
\node[fill] (a1) at (90:0) {};
\node (a2) at (90:1) {};
\node[fill] (a3) at (90:2) {};
\node (a4) at (210:1) {};
\node[fill] (a5) at (210:2) {};
\node (a6) at (330:1) {};
\node[fill] (a7) at (330:2) {};
\node (a8) at (150:1) {};
\node (a9) at (270:1) {};
\node (a10) at (30:1) {};
\draw[->] (a3) to [bend right=20] (a9);
\draw[->]  (a3)--(a6);
\draw[->]  (a5) to [bend right=20] (a10);
\draw[->]  (a5)--(a2);
\draw[->]  (a7) to  [bend right=20] (a8);
\draw[->]  (a7)--(a4);
\draw[->]  (a1)--(a9);
\draw[->]  (a1)--(a10);
\draw[->]  (a1)--(a8);
\draw[->]  (a3)--(a4);
\draw[->]  (a7)--(a2);
\draw[->]  (a5)--(a6);
\end{tikzpicture}
\medskip\hfil\begin{tikzpicture}[scale=0.7]
\node[fill] (a1) at (90:0) {};
\node (a2) at (90:1) {};
\node[fill] (a3) at (90:2) {};
\node (a4) at (210:1) {};
\node[fill] (a5) at (210:2) {};
\node (a6) at (330:1) {};
\node[fill] (a7) at (330:2) {};
\node (a8) at (150:1) {};
\node (a9) at (270:1) {};
\node (a10) at (30:1) {};
\draw (a2) to [loop above] (a2); 
\draw (a4) to [loop above] (a4); 
\draw (a6) to [loop above] (a6); 
\draw (a8) to [loop above] (a8); 
\draw (a9) to [loop above] (a9); 
\draw (a10) to [loop above] (a10); 
\end{tikzpicture}
\medskip\hfil\begin{tikzpicture}[scale=0.7]
\node[fill] (a1) at (90:0) {};
\node (a2) at (90:1) {};
\node[fill] (a3) at (90:2) {};
\node (a4) at (210:1) {};
\node[fill] (a5) at (210:2) {};
\node (a6) at (330:1) {};
\node[fill] (a7) at (330:2) {};
\node (a8) at (150:1) {};
\node (a9) at (270:1) {};
\node (a10) at (30:1) {};
\draw[<-] (a3)--(a2);
\draw[<-] (a5)--(a4);
\draw[<-] (a7)--(a6);
\draw[<-] (a3)--(a8);
\draw[<-] (a5)--(a8);
\draw[<-] (a5)--(a9);
\draw[<-] (a7)--(a9);
\draw[<-] (a7)--(a10);
\draw[<-] (a3)--(a10);
\draw[<-] (a1)--(a2);
\draw[<-] (a1)--(a4);
\draw[<-] (a1)--(a6);
\end{tikzpicture}
\medskip\hfil\begin{tikzpicture}[scale=0.7]
\node[fill] (a1) at (90:0) {};
\node (a2) at (90:1) {};
\node[fill] (a3) at (90:2) {};
\node (a4) at (210:1) {};
\node[fill] (a5) at (210:2) {};
\node (a6) at (330:1) {};
\node[fill] (a7) at (330:2) {};
\node (a8) at (150:1) {};
\node (a9) at (270:1) {};
\node (a10) at (30:1) {};
\draw (a2)--(a4)--(a6)--(a2);
\draw (a2)--(a8)--(a4)--(a9)--(a6)--(a10)--(a2);
\draw (a8)--(a9)--(a10)--(a8);
\end{tikzpicture}
\medskip\hfil\begin{tikzpicture}[scale=0.7]
\node[fill] (a1) at (90:0) {};
\node (a2) at (90:1) {};
\node[fill] (a3) at (90:2) {};
\node (a4) at (210:1) {};
\node[fill] (a5) at (210:2) {};
\node (a6) at (330:1) {};
\node[fill] (a7) at (330:2) {};
\node (a8) at (150:1) {};
\node (a9) at (270:1) {};
\node (a10) at (30:1) {};
\draw[<-] (a3) to [bend right=20] (a9);
\draw[<-] (a3)--(a6);
\draw[<-] (a5) to [bend right=20] (a10);
\draw[<-] (a5)--(a2);
\draw[<-] (a7) to  [bend right=20] (a8);
\draw[<-] (a7)--(a4);
\draw[<-] (a1)--(a9);
\draw[<-] (a1)--(a10);
\draw[<-] (a1)--(a8);
\draw[<-] (a3)--(a4);
\draw[<-] (a7)--(a2);
\draw[<-] (a5)--(a6);
\end{tikzpicture}
\medskip\hfil\begin{tikzpicture}[scale=0.6]
\node[fill] (a1) at (90:0) {};
\node (a2) at (90:1) {};
\node[fill] (a3) at (90:2) {};
\node (a4) at (210:1) {};
\node[fill] (a5) at (210:2) {};
\node (a6) at (330:1) {};
\node[fill] (a7) at (330:2) {};
\node (a8) at (150:1) {};
\node (a9) at (270:1) {};
\node (a10) at (30:1) {};
\draw (a2) to [bend left=20] (a9);
\draw (a10) to [bend left=20] (a4);
\draw (a6) to [bend left=20] (a8);
\end{tikzpicture}
\medskip
\caption{Coherent configuration of rank $9$ determined by the subdivision of the complete graph $K_4$.}
\label{fig:coherent}
\end{figure}

Let $\Gamma = (X,E)$ be a distance biregular-graph with bipartition $(X',X'')$, and let $d'=\max\{d(x,y) \mid x\in X', y\in X\}$ and $d''=\max\{d(x,y) \mid x\in X'', y\in X\}$. 
For $i\geq 0$, we define the relations $R'_{i},R''_{i}$ on $X$ by
\[
R'_{i} = \{(x,y)\in X\times X \mid x\in X' ,\ d(x,y) = i\},
\]
\[
R''_{i} = \{(x,y)\in X\times X \mid x\in X'' ,\ d(x,y) = i\}.
\]
Then  $R'_{i} = \emptyset$ for $i>d'$ and $R''_{i} = \emptyset$ for $i>d''$.  
Clearly, $R'_i\cap R'_j = \emptyset$ and $R''_i\cap R''_j = \emptyset$ for all $i,j$. 
Since $X'\cap X'' = \emptyset$, then $R'_i\cap R''_j = \emptyset$ for all $i,j$. 
Moreover, 
\[
R'_0 \cup R'_1\cup \cdots \cup R'_{d'} = X' \times X,
\]
\[
R''_0 \cup R''_1\cup \cdots \cup R''_{d''} = X'' \times X.
\]
Since $X'\cup X'' = X$, then $\{R'_0,\dots,R'_{d'},R''_0,\dots,R''_{d''}\}$ is a partition of $X\times X$. 
Furthermore, $R'_0 =\{(x,x) \mid x\in X'\}$ and $R''_0 = \{(x,x) \mid x\in X''\}$,  and so $R'_0 \cup R''_0 = \{(x,x) \mid x\in X\}$. 
Since $\Gamma$ is bipartite, then $R'_{2i}$ and $R''_{2i}$ are symmetric, while $(x,y) \in R'_{2i+1}$ if and only if $(y,x)\in R''_{2i+1}$ for all $i$.  
In particular, $R'_1 \cup R''_1 = E$. 
Finally, since $\Gamma$ is distance-biregular, there exist numbers $p'_{ij}(k)$ and $p''_{ij}(k)$ such that for all $(x,y)$ with $d(x,y)=k$, we have
\[
|\Gamma_i(x)\cap \Gamma_j(y)| = 
\begin{cases}
    p'_{ij}(k)  & \textnormal{if $x\in X'$},\\
    p''_{ij}(k) & \textnormal{if $x\in X''$}
\end{cases}
\]
(see, e.g. \cite[Lemma 3.10]{A}). 
Therefore, $\Pi_{\Gamma} = \{R'_0,\dots,R'_{d'},R''_0,\dots,R''_{d''}\}$ is the coherent configuration of $\Gamma$. 

The intersection numbers $p'_{ij}(k)$ and $p''_{ij}(k)$ of $\Pi_\Gamma$ can be computed recursively in terms of the intersection arrays of the distance-biregular graph $\Gamma$ (see \cite[Lemma 6.3.1]{L2}).

\begin{lemma}[\cite{L2}]\label{lem:int}
\[
p'_{i0}(k) = \delta_{ik}, \qquad p'_{0j}(k)  = \delta_{jk},
\]
\[
p'_{i1}(k) =
\begin{cases}
    c'_i & \textnormal{if $k = i+1$},\\
    b'_i & \textnormal{if $k = i-1$},\\
    0 & \textnormal{otherwise}.
\end{cases} 
\] 
\[
p'_{ij+1}(2k)
= \frac{1}{c'_{j+1}} 
[p'_{i-1 j}(2k)b'_{i-1}  + p'_{i+1j}(2k)c'_{i+1} - p'_{ij-1}(2k) b'_{j-1}
],
\]
\[
p'_{ij+1}(2k+1) 
= \frac{1}{c''_{j+1}} [p'_{i-1 j}(2k+1)b'_{i-1} + p'_{i+1j}(2k+1)c'_{i+1} - p'_{ij-1}(2k+1) b''_{j-1}].
\]
Similarly,
\[
p''_{i0}(k) = \delta_{ik}, \qquad p''_{0j}(k)  = \delta_{jk},
\]
\[
p''_{i1}(k) =
\begin{cases}
    c''_i & \textnormal{if $k = i+1$},\\
    b''_i & \textnormal{if $k = i-1$},\\
    0 & \textnormal{otherwise}.
\end{cases} 
\] 
\[
p''_{ij+1}(2k)
= \frac{1}{c''_{j+1}} 
[p''_{i-1 j}(2k)b''_{i-1}  + p''_{i+1j}(2k)c''_{i+1} - p''_{ij-1}(2k) b''_{j-1}
],
\]
\[
p''_{ij+1}(2k+1) 
= \frac{1}{c'_{j+1}} [p''_{i-1 j}(2k+1)b''_{i-1} + p''_{i+1j}(2k+1)c''_{i+1} - p''_{ij-1}(2k+1) b'_{j-1}].
\]
\end{lemma}

Therefore, if two distance-biregular graphs have the same intersection arrays, then their coherent configurations are intersection equivalent, and hence  the two graphs are $C^3$-equivalent. 

\begin{theorem} \label{thm:cospec}
Two distance-biregular graphs $\Gamma$ and $\Delta$ are cospectral if and only if they are  $C^3$-equivalent.
\end{theorem}
\begin{proof}
If $\Gamma$ and $\Delta$ are cospectral, then Corollary \ref{cor:arrays} implies that $\Gamma$ and $\Delta$ have the same intersection arrays. 
From Lemma \ref{lem:int}, it follows that the coherent configurations $\Pi_\Gamma$ and $\Pi_\Delta$ have the same intersection numbers, and so they are intersection equivalent. 
By Theorem \ref{thm:cfi}, since $\Pi_\Gamma$ and $\Pi_\Delta$ are intersection equivalent, then  $\Gamma$ and $\Delta$ are $C^3$-equivalent. 
The converse also holds, as we already had Theorem \ref{thm:C3}.
\end{proof}

For distance-regular graphs, a similar result was  proved by Man{\v{c}}inska,  Roberson and Varvitsiotis (see \cite[Theorem 7.6]{MRV}). 
To stress the connection with logical definability, we state their result as follows. 

\begin{theorem}[\cite{MRV}]\label{thm:drg}
Two distance-regular graphs are cospectral if and only if they are $C^3$-equivalent. 
\end{theorem}

As we mentioned before, Godsil and Shawe-Taylor proved in \cite{GS} that there are no other distance-regularized graphs apart from the ones which are distance-regular or distance-biregular. 

\begin{theorem}[\cite{GS}]
    A distance-regularized graph is either distance-regular or distance-biregular.
\end{theorem}

Thus the next result follows immediately. 

\begin{theorem}
    Two distance-regularized graphs are cospectral if and only if they are $C^3$-equivalent. 
\end{theorem}

\section*{Acknowledgments}
This work started during a Dagstuhl Seminar on Logic and Random Discrete Structures, Germany, in February 2022. The authors would like to thank the organizers of the workshop. A preliminary version of this work appeared in the proceedings of the XII Latin-American Algorithms, Graphs and Optimization Symposium (LAGOS 2023).  

Aida Abiad is supported by the Dutch Research Council through the grant VI.Vidi.213.085. The research of Octavio Zapata is partially supported by the Mexican Research Council through the grant SNII 620178.

\end{document}